\documentclass{birkjour}
\newtheorem{theorem}{Theorem}[section]

\newtheorem{definition}[theorem]{Definition}

\newtheorem{proposition}[theorem]{Proposition}
\newtheorem{remark}[theorem]{Remark}

\newcommand{\R}{\mathbb{R}}

\newcommand{\pqbinomial}[4]{\mbox{$
\biggl[ \!\!
\begin{array}{c}
#1\\
 #2
\end{array}\!\!\biggr]_{
\!{#3,#4}} $} }

\begin{document}

%
%
%
%
%
%
%
%
%

\title[$(p,q)$-Gamma function]
 {On the $(p,q)$-Gamma\\ and the $(p,q)$-Beta functions}

\author[P. Njionou Sadjang]{P. Njionou Sadjang}

\address{%
Faculty of Industrial Engineering\\
University of Douala\\
Cameroon}

\email{pnjionou@yahoo.fr}

\thanks{This work was supported  by TWAS }


\date{\today}

\begin{abstract}
We introduce  new generalizations of the Gamma and the Beta functions. Their properties are investigated and kown results are obtained as particular cases. 
\end{abstract}

\maketitle
\section{Introduction}

The $q$-deformed algebras \cite{odzijewicz, quesne} and their generalizations ($(p,q)$-deformed algebras) \cite{burban, chakrabarti}  attract much attention these last years. The main reason is that these topics stand for a meeting point of today's fast developping areas in methematics and physics like the theory of quantum orthogonal polynomials and special functions, quantum groups, conformal field theories and statistics. From these works, many generalizations of special functions arise. There is a considerable list of references.

In this work, we give a new generalization of the Gamma and the Beta functions, namely, the $(p,q)$-Gamma and the $(p,q)$-Beta functions. Their main properties are stated and proved and connections with the previous work are done.   It is to be noted that in \cite{KM}, the authors provided another generalization of the Gamma function. 

The Euler gamma function $\Gamma(x)$ first happens in 1729 in a correspondance between Euler and Goldbach and  is defined for $x >0$ by (see \cite{Andrews-Askey, Nico})
\begin{equation}
\Gamma(x)=\int_{0}^{\infty}t^{x-1}e^{-t}dt.
\end{equation}
Euler gave an equivalent representation of the Gamma function \cite{Andrews-Askey,apostol, sandor, Nico}
\begin{equation}
\Gamma(x)=\lim\limits_{n\to \infty}\dfrac{n!n^x}{x(x+1)\cdots(x+n)}.
\end{equation}
An immediate consequence of these representations is 
\[\Gamma(x+1)=x\Gamma(x).\]
Also, for any nonnegative integer $n$,
\[\Gamma(n+1)=n!\]
follows from the above argument.

Closely connected with the Gamma function is the Beta function \cite{Andrews-Askey, Nico} defined as 
\begin{equation}
B(x,y)=\int_0^1t^{x-1}(1-t)^{y-1}dt, \;\;\Re x>0,\;\Re y>0.
\end{equation}
The Beta function is related to the Gamma function in the following way
\begin{equation}\label{betagamma}
B(x,y)=\dfrac{\Gamma(x)\Gamma(y)}{\Gamma(x+y)}.
\end{equation}

Jackson (see \cite{kim1,kim2,kim3, kim4}) defined the $q$-analogue of the gamma function as

\begin{equation}\label{qgamma1}
\Gamma_q(x)=\dfrac{(q;q)_{\infty}}{(q^x;q)_{\infty}}(1-q)^{1-x},\;\; 0<q<1,
\end{equation}
and
\begin{equation}
\Gamma_q(x)=\dfrac{(q^{-1};q^{-1})_{\infty}}{(q^{-x};q^{-1})_{\infty}}(q-1)^{1-x}q^{\binom{x}{2}},\;\; q>1.
\end{equation}

An equivalent definition of (\ref{qgamma1}) is given in \cite{kac} as 
\begin{equation}
\Gamma_q(x)=\int_{0}^{\infty}t^{x-1}E_q^{-qt}d_qt
\end{equation}
where the $q$-integral is defined by (see \cite{foupouagnignithesis, kac, njionou2013}) 
\begin{equation}\label{qinteq2}
\int_0^af(s)d_qs=a(1-q)\sum\limits_{n=0}^{\infty}
q^nf(aq^n),\quad a>0,
\end{equation}
\begin{equation}\label{qinteq3}
\int_a^0f(s)d_qs=-a(1-q)\sum\limits_{n=0}^{\infty}
q^nf(aq^n), \quad a<0,
\end{equation}
\begin{equation}\label{qinteq4}
\int_a^\infty
f(s)d_qs=a(q^{-1}-1)\sum\limits_{n=0}^{\infty}
q^{-n}f(aq^{-n-1}),\quad a>0,
\end{equation}
\begin{equation}\label{qinteq5}
\int_{-\infty}^a
f(s)d_qs=-a(q^{-1}-1)\sum\limits_{n=0}^{\infty}
q^{-n}f(aq^{-n-1}),\quad a<0,
\end{equation}
and can be  extended to the whole real line by using
relations (\ref{qinteq2})-(\ref{qinteq5}) and the following rules
\begin{eqnarray*}
  \int_a^bf(s)d_qs &=& \int_a^0f(s)d_qs+\int_0^bf(s)d_qs\quad \forall a,b\in\R  \\
  \int_a^\infty f(s)d_qs &=& \int_a^bf(s)d_qs+\int_b^\infty f(s)d_qs\quad \forall a,b\in\R,\,\,a<0,\,b>0 \\
  \int_{-\infty}^bf(s)d_qs &=&\int_{-\infty}^af(s)d_qs+\int_a^bf(s)d_qs\quad \forall a,b\in\R,\,\, a<0,\,b>0  \\
  \int_{-\infty}^\infty f(s)d_qs &=&
  \int_{-\infty}^af(s)d_qs+\int_a^bf(s)d_qs+\int_b^\infty
  f(s)d_qs\quad\forall a,b\in\R.
\end{eqnarray*}

\section{Definitions and Miscellaneous Relations}

Let us introduce the following notation (see  \cite{JS2006},\cite{  JR2010},\cite{ njionou-2014-3})
\begin{equation}\label{pqnumber}
[n]_{p,q}=\frac{p^n-q^n}{p-q},
\end{equation}
for any positive integer. 

\noindent The twin-basic number is a natural generalization of the $q$-number, that is
\begin{equation}
\lim\limits_{p\to 1}[n]_{p,q}=[n]_q.
\end{equation}

\noindent The $(p,q)$-factorial is defined by  (\cite{JR2010, njionou-2014-3})
\begin{equation}
[n]_{p,q}!=\prod_{k=1}^{n}[k]_{p,q}!,\quad n\geq 1,\quad [0]_{p,q}!=1.
\end{equation}
\noindent Let us introduce also the so-called $(p,q)$-binomial coefficient
\begin{equation}\label{pqbin}
\pqbinomial{n}{k}{p}{q}=\dfrac{[n]_{p,q}!}{[k]_{p,q}![n-k]_{p,q}!}, \quad 0\leq k\leq n.
\end{equation}
 Note that as $p\to 1$, the $(p,q)$-binomial coefficients reduce to the $q$-binomial coefficients.\\
 It is clear by definition that 
 \begin{equation}\label{binomial1}
 \pqbinomial{n}{k}{p}{q}=\pqbinomial{n}{n-k}{p}{q}.
 \end{equation}
 
\noindent  Let us introduce also the so-called the $(p,q)$-powers \cite{njionou-2014-3}
\begin{eqnarray}
 (x\ominus a)_{p,q}^n&=&(x-a)(px-aq)\cdots (xp^{n-1}-aq^{n-1}),\\
 (x\oplus a)_{p,q}^n&=&(x+a)(px+aq)\cdots (xp^{n-1}+aq^{n-1}).
\end{eqnarray}

\noindent These definition are extended to 
\begin{eqnarray}\label{infinitpq}
(a\ominus b)_{p,q}^{\infty}=\prod_{k=0}^{\infty}(ap^k-q^kb)\\
(a\oplus b)_{p,q}^{\infty}=\prod_{k=0}^{\infty}(ap^k+q^kb)
\end{eqnarray}
where the convergence is required.

\begin{proposition}
The following identities are easily verified
\begin{eqnarray}
(a\ominus b)_{p,q}^{n}&=&\dfrac{(a\ominus b)_{p,q}^{\infty}}{(ap^n\ominus bq^n)_{p,q}^{\infty}}\\
(a\ominus b)_{p,q}^{n+k}&=&(a\ominus b)_{p,q}^n(ap^n\ominus bq^n)_{p,q}^{k}\\
(ap^n\ominus bq^n)_{p,q}^k&=&\dfrac{(a\ominus b)_{p,q}^k(ap^k\ominus bq^k)_{p,q}^n}{(a\ominus b)_{p,q}^n}\\
(ap^k\ominus bq^k)_{p,q}^{n-k}&=&\frac{(a\ominus b)_{p,q}^{n}}{(a\ominus b)_{p,q}^{k}}\\
(ap^{2k}\ominus bq^{2k})_{p,q}^{n-k}&=&\frac{(a\ominus b)_{p,q}^{n}(ap^n\ominus bq^n)_{p,q}^k}{(a\ominus b)_{p,q}^{2k}}\\
(a^2\ominus b^2)_{p,q}^n&=& (a\ominus b)^n_{p,q}(a\oplus b)_{p,q}^n\\
(a\ominus b)_{p,q}^{2n}&=& (a\ominus b)_{p^2,q^2}^n(ap\ominus bq)_{p^2,q^2}^n\\
(a\ominus b)_{p,q}^{3n}&=& (a\ominus b)_{p^3,q^3}^n(ap\ominus bq)_{p^3,q^3}^n(ap^2\ominus bq^2)_{p^3,q^3}^n\\
(a\ominus b)_{p,q}^{\ell n}&=& \prod_{j=0}^{\ell-1}(ap^j\ominus bq^j)_{p^\ell,q^\ell}^n.
\end{eqnarray} 
\end{proposition}

\section{The $(p,q)$-Gamma function}


\begin{definition}\label{pqgamma}
Let $x$ be a complex number, we define the  $(p,q)$-Gamma function as
\begin{equation}\label{pqgam}
\Gamma_{p,q}(x)=\dfrac{(p\ominus q)_{p,q}^{\infty}}{(p^x\ominus q^x)^\infty_{p,q}}(p-q)^{1-x},\; 0<q<p.
\end{equation}
\end{definition}

\begin{remark}
Note that in (\ref{pqgam}), if we put $p=1$, then $\Gamma_{p,q}$ reduces to $\Gamma_q$.
\end{remark}


\begin{proposition}
The $(p,q)$-Gamma function fulfils the following fundemental relation
\begin{equation}\label{gammarec}
\Gamma_{p,q}(x+1)=[x]_{p,q}\Gamma_{p,q}(x).
\end{equation}

\end{proposition}

\begin{proof}
By definition, we have 
\begin{align*}
\Gamma_{p,q}(x+1)=\dfrac{(p\ominus q)_{p,q}^{\infty}}{(p^{x+1}\ominus q^{x+1})^\infty_{p,q}}(p-q)^{-x}
\end{align*}

\end{proof}

\begin{remark}
If $n$ is a nonnegative integer, it follows from (\ref{gammarec}) that 
\[\Gamma_{p,q}(n+1)=[n]_{p,q}!.\]
It can be also easyly seen from the definition that
\[\Gamma_{p,q}(n+1)=\dfrac{(p\ominus q)^n_{p,q}}{(p-q)^n}.\]
\end{remark}

\begin{proposition}[$(p,q)$-Legendre's multiplication formula] The following multiplication formula applies
\begin{equation}
\Gamma_{p,q}(2x)\Gamma_{p^2,q^2}\left(\frac{1}{2}\right)=(p+q)^{2x-1}\Gamma_{p^2,q^2}(x)\Gamma_{p^2,q^2}\left(x+\frac{1}{2}\right).
\end{equation}
\end{proposition}

\begin{proof}
From the definition, we have
\begin{align*}
\Gamma_{p^2,q^2}(x)&=\dfrac{(p^2\ominus q^2)_{p^2,q^2}^{\infty}}{(p^{2x}\ominus q^{2x})_{p^2,q^2}^{\infty}}(p^2-q^2)^{1-x}\\
\Gamma_{p^2,q^2}\left(x+\frac{1}{2}\right)&= \dfrac{(p^2\ominus q^2)_{p^2,q^2}^{\infty}}{(p^{2x+1}\ominus q^{2x+1})_{p^2,q^2}^{\infty}}(p^2-q^2)^{\frac{1}{2}-x}\\
\Gamma_{p^2,q^2}\left(\frac{1}{2}\right)&= \dfrac{(p^2\ominus q^2)_{p^2,q^2}^{\infty}}{(p\ominus q)_{p^2,q^2}^{\infty}}(p^2-q^2)^{\frac{1}{2}}.
\end{align*}
Hence, 
\begin{align*}
\dfrac{\Gamma_{p^2,q^2}(x)\Gamma_{p^2,q^2}\left(x+\frac{1}{2}\right)}{\Gamma_{p^2,q^2}\left(\frac{1}{2}\right)}&=
\dfrac{(p^2\ominus q^2)_{p^2,q^2}^{\infty}(p\ominus q)_{p^2,q^2}^{\infty}}{(p^{2x}\ominus q^{2x})_{p^2,q^2}^{\infty}(p^{2x+1}\ominus q^{2x+1})_{p^2,q^2}^{\infty}}(p^2-q^2)^{1-2x}\\
&=\dfrac{(p\ominus q)_{p,q}^{\infty}}{(p^{2x}\ominus q^{2x})_{p,q}^{\infty}}(p-q)^{1-2x}(p+q)^{1-2x}\\
&=(p+q)^{1-2x}\Gamma_{p,q}(2x).
\end{align*}
This proves the proposition.
\end{proof}
\noindent The $(p,q)$-Legendre's multiplication formula is generalized as follows.
\begin{proposition}[$(p,q)$-Gauss' multiplication formula]
The following multiplication formula applies
\begin{align}
\Gamma_{p,q}(nx)\prod_{k=1}^{n-1}\Gamma_{p^n,q^n}\left(\frac{k}{n}\right)=([n]_{p,q})^{nx-1}\prod_{k=0}^{n-1}\Gamma_{p^n,q^n}\left(x+\frac{k}{n}\right).
\end{align}
\end{proposition}

\begin{proof}
As for the previous proposition, we start by using the definition as follows
\begin{align*}
\Gamma_{p^n,q^n}\left(\frac{k}{n}\right)&=\dfrac{(p^n\ominus q^n)^{\infty}_{p^n,q^n}}{(p^k\ominus q^k)^{\infty}_{p^n,q^n}}(p^n-q^n)^{1-\frac{k}{n}},\\
\Gamma_{p^n,q^n}\left(x+\frac{k}{n}\right)&=\dfrac{(p^n\ominus q^n)^{\infty}_{p^n,q^n}}{(p^{nx+k}\ominus q^{nx+k})^{\infty}_{p^n,q^n}}(p^n-q^n)^{1-\frac{k}{n}-x}.
\end{align*}
Hence, we have
\begin{align*}
\prod_{k=1}^{n-1}\Gamma_{p^n,q^n}\left(\frac{k}{n}\right)&=\dfrac{\left[ (p^n\ominus q^n)^{\infty}_{p^n,q^n} \right]^{n-1}}{\prod\limits_{k=1}^{n-1}(p^k\ominus q^k)^{\infty}_{p^n,q^n}}(p^n-q^n)^{\sum\limits_{k=1}^{n-1}(1-\frac{k}{n})}\\
&=\dfrac{\left[ (p^n\ominus q^n)^{\infty}_{p^n,q^n} \right]^{n}}{\prod\limits_{k=0}^{n-1}(p.p^k\ominus q.q^k)^{\infty}_{p^n,q^n}}(p^n-q^n)^{\frac{n-1}{2}}\\
&=\dfrac{\left[ (p^n\ominus q^n)^{\infty}_{p^n,q^n} \right]^{n}}{(p\ominus q)^{\infty}_{p,q}}(p^n-q^n)^{\frac{n-1}{2}}
\end{align*}
and 
\begin{align*}
\prod_{k=0}^{n-1}\Gamma_{p^n,q^n}\left(x+\frac{k}{n}\right)&=\dfrac{\left[(p^n\ominus q^n)^{\infty}_{p^n,q^n}\right]^n}{\prod\limits_{k=0}^{n-1}(p^{nx+k}\ominus q^{nx+k})^{\infty}_{p^n,q^n}}(p^n-q^n)^{\sum\limits_{k=0}^{n-1}(1-\frac{k}{n}-x)}\\
&=\dfrac{\left[(p^n\ominus q^n)^{\infty}_{p^n,q^n}\right]^n}{(p^{nx}\ominus q^{nx})^{\infty}_{p,q}}(p^n-q^n)^{(\frac{n-1}{2}+1-nx)}.
\end{align*}
It follows that 
\begin{align*}
\dfrac{\prod\limits_{k=0}^{n-1}\Gamma_{p^n,q^n}\left(x+\frac{k}{n}\right)}{\prod\limits_{k=1}^{n-1}\Gamma_{p^n,q^n}\left(\frac{k}{n}\right)}&= \dfrac{(p\ominus q)^{\infty}_{p,q}}{(p^{nx}\ominus q^{nx})^{\infty}_{p,q}}(p^n-q^n)^{1-nx}\\
&= \dfrac{(p\ominus q)^{\infty}_{p,q}}{(p^{nx}\ominus q^{nx})^{\infty}_{p,q}}(p-q)^{1-nx}\left(\dfrac{p^n-q^n}{p-q}\right)^{1-nx}\\
&=([n]_{p,q})^{1-nx}\Gamma_{p,q}(nx).
\end{align*}
The proposition is then proved.
\end{proof}


\section{The $(p,q)$-Beta function}

\begin{definition}
Following \eqref{betagamma}, we define the $(p,q)$-Beta functions as 
\begin{equation}
B_{p,q}(x,y)=\frac{\Gamma_{p,q}(x)\Gamma_{p,q}(y)}{\Gamma_{p,q}(x+y)}.
\end{equation} 
\end{definition}

\begin{proposition}
The $(p,q)$-Beta function fulfils the following properties
\begin{eqnarray}
B_{p,q}(x,y+1)&=& \dfrac{[y]_{p,q}}{[x+y]_{p,q}}B_{p,q}(x,y)\\
B_{p,q}(x+1,y)&=& \dfrac{[x]_{p,q}}{[x+y]}_{p,q}B_{p,q}(x,y)\\
B_{p,q}(x+1,y)&=&\dfrac{[x]_{p,q}}{[y]_{p,q}}B_{p,q}(x,y+1)\\
B_{p,q}(x+n,y)&=& \dfrac{(p^x\ominus q^x)^n_{p,q}}{(p^{x+y}\ominus q^{x+y})^n_{p,q}}B_{p,q}(x,y).
\end{eqnarray}
\end{proposition}

\end{document}